\documentclass[11pt]{amsart}
\usepackage{amsmath,amssymb,amsthm}
\usepackage{mathrsfs}
\usepackage{enumerate}
\usepackage{comment}

\newtheorem{introthm}{Theorem}

\newtheorem{theorem}{Theorem}[section]
\newtheorem{proposition}[theorem]{Proposition}
\newtheorem{lemma}[theorem]{Lemma}
\newtheorem{corollary}[theorem]{Corollary}
\theoremstyle{definition}

\theoremstyle{remark}
\newtheorem{remark}[theorem]{Remark}

\newcommand{\R}{\mathbb{R}}
\newcommand{\E}{\mathbb{E}}

\newcommand{\N}{\mathbb{N}}
\newcommand{\Prob}{\mathscr{P}}
\DeclareMathOperator{\supp}{supp}
\DeclareMathOperator{\Opt}{Opt}
\DeclareMathOperator{\Id}{Id}

\DeclareMathOperator{\bary}{bar}
\newcommand{\DirB}{\mathbf{Dir}}
\newcommand{\TanB}{\mathbf{Tan}}

\begin{document}

\title{Rigidity of Wasserstein spaces over Riemannian manifolds}

\author{David Lenze}

\address
  {Karlsruher Institut f\"ur Technologie\\ Fakult\"at f\"ur Mathematik \\
Englerstr. 2 \\
76131 Karlsruhe,
Germany}
\email{david.lenze@kit.edu}

\begin{abstract} We show that $L^2$ Wasserstein spaces over Riemannian manifolds are isometrically rigid if and only if their underlying Riemannian manifolds do not admit a Euclidean de Rham factor. We further show that, unless the manifold is isometric to the real line, every isometry of the Wasserstein space is shape-preserving in the sense of Kloeckner. 
Finally, we demonstrate that two such Wasserstein spaces are isometric if and only if their underlying Riemannian manifolds are isometric.

\end{abstract}
\maketitle

\renewcommand{\theequation}{\arabic{section}.\arabic{equation}}
\pagenumbering{arabic}

\section{Introduction}

Let $(M,g)$ be a smooth, connected, and complete Riemannian manifold with distance function $d$ induced by $g$.

We denote by $ \mathscr P_2(M)$ the space of Borel probability measures $\mu$ on $M$ with finite second moment, i.e.\
\begin{equation*}
\int_M d^2(x, x_0) \, d\mu(x) < \infty \quad \text{for some (and hence all) } x_0 \in M.
\end{equation*}

The \emph{Wasserstein distance} $W_2$ on $ \mathscr P_2(M)$ is defined by
\begin{equation}
W_2(\mu, \nu) = \left( \inf_{\pi \in \Pi(\mu, \nu)} \int_{M \times M} d^2(x, y) \, d\pi(x, y) \right)^{1/2},
\end{equation}
where $\Pi(\mu, \nu)$ denotes the set of \emph{transport plans} between $\mu$ and $\nu$, i.e.\ Borel probability measures on $M \times M$ satisfying $(p_1)_\# \pi = \mu$ and $(p_2)_\# \pi = \nu$ for the natural projections $p_1$ and $p_2$.
The Wasserstein distance $W_2$ induces a metric on $\mathscr P_2(M)$, and $(\mathscr P_2(M),W_2)$ is geodesic and Polish.

For an isometry $\gamma:M \to M$, the push-forward $\gamma_\#: \mathscr P_2(M)\to \mathscr P_2(M)$ is an isometry with respect to the Wasserstein metric $W_2$ (cf.~\cite{MR2459454}). This naturally raises the question of whether all isometries arise in this way---that is, whether the following embedding is surjective:
\begin{equation}\label{eq:embedding}
\operatorname{Isom}(M) \hookrightarrow \operatorname{Isom}(\mathscr P_2(M)), \quad \gamma \mapsto \gamma_\#.
\end{equation}

For the Euclidean case $M=\mathbb{E}^m$, the answer is \emph{negative} as demonstrated by Kloeckner~\cite{MR2731158}: for $m\geq 2$, every isometry $\Phi$ of $\mathscr P_2(M)$ is \emph{shape-preserving}: for each $\mu$ there is a $\mu$-dependent $\gamma_\mu \in \operatorname{Isom}(M)$ with $\Phi(\mu)=(\gamma_\mu)_\#\mu$. Interestingly, the one-dimensional case $m=1$ is even more striking: the Wasserstein space admits non-shape-preserving, \emph{exotic} isometries that shift mass distributions in a way that has no geometric counterpart on the base space.

Wasserstein spaces over various Riemannian manifolds have since been studied. In contrast to the Euclidean setting, these are generally \emph{isometrically rigid}---meaning \eqref{eq:embedding} is an isomorphism. Indeed, rigidity has been established for manifolds of strictly negative sectional curvature \cite{MR3509929}, strictly positive sectional curvature \cite{MR4469075}, and specialized cases like tori and spheres \cite{MR4516798}. Conversely, extending Kloeckner's results, a recent study \cite{MR5039553} showed that if the base manifold splits off a Euclidean factor, the corresponding Wasserstein space is \emph{never isometrically rigid}.

Recall that by the general non-simply connected form of de Rham's decomposition theorem
\cite{MR1473665}, any complete connected Riemannian manifold decomposes uniquely as
a direct product $M_0 \times M_1 \times \dots \times M_p$, where $M_0$ is the
maximal Euclidean factor and the remaining factors are indecomposable.

In this paper, we completely resolve the question of isometric rigidity for
Wasserstein spaces over Riemannian manifolds:

\begin{introthm}\label{main}
Let $M$ be a complete connected Riemannian manifold. Then $\operatorname{Isom}(\mathscr P_2(M)) \cong \operatorname{Isom}(M)$ if and only if $M$ does not split a Euclidean factor.
\end{introthm}

Theorem~\ref{main} settles \emph{when} \eqref{eq:embedding} is surjective, but
leaves open the structure of $\operatorname{Isom}(\mathscr P_2(M))$ when $M$ does
split a Euclidean factor. Our second main result addresses this: away
from the single exceptional base $M\cong \mathbb E$ the additional isometries are \emph{never exotic}, but merely rotate each measure
about the barycentre of its Euclidean marginal.

\begin{introthm}\label{shape}
Let $M$ be a complete connected Riemannian manifold with $M\not\cong\mathbb E$,
and let $\mathbb E^k$ be its maximal Euclidean de Rham factor. Then every isometry
of $\mathscr P_2(M)$ is shape-preserving; more precisely,
\[
   \operatorname{Isom}\big(\mathscr P_2(M)\big)\ \cong\
   \operatorname{Isom}(M)\ltimes O(k),
\]
where $O(k)$ acts by rotating each measure about the barycentre of its
$\mathbb E^k$-marginal. In particular, $\mathscr P_2(M)$ admits non-shape-preserving
(exotic) isometries if and only if $M\cong\mathbb E$.
\end{introthm}

For $k=0$, Theorem~\ref{shape} reduces to Theorem~\ref{main},
and for $M=\mathbb E^m$, $m\geq 2$ it recovers Kloeckner's
$\operatorname{Isom}(\mathscr P_2(\mathbb E^m))\cong\operatorname{Isom}(\mathbb
E^m)\ltimes O(m)$. Its force lies in the mixed case $M=N\times
\mathbb E^k$ with $N$ non-trivial and free of Euclidean factors, where \cite{MR5039553} established non-rigidity but left open whether every isometry is nevertheless shape-preserving. Most compellingly, for $k=1$, as soon as \emph{another} indecomposable factor is present, the exotic isometries of the line completely disappear, and only 
barycentre rotations survive.

Finally, we derive the following rigidity statement:

\begin{introthm}\label{B}
   Let $M$ and $N$ be complete connected Riemannian manifolds. The $L^2$-Wasserstein spaces $(\mathscr P_2(M),W_2)$ and $(\mathscr P_2(N),W_2)$ are isometric if and only if $M$ is isometric to $N$.
\end{introthm}

This should be compared to \cite[Theorem 1.2]{MR3509929}, where an analogous statement was shown for locally compact geodesic metric spaces that are \emph{geodesically complete}, here meaning that each minimizing geodesic segment can be extended to a minimizing line. In the Riemannian setting this already fails for any compact space but also for non-compact ones like two-dimensional paraboloids of revolution, all of which Theorem~\ref{B} still covers.

Isometric rigidity has also been widely studied beyond smooth Riemannian manifolds: for singular spaces of negative curvature \cite{MR3509929}, for Hilbert spaces~\cite{MR4524213}, countable discrete spaces~\cite{MR4000104} and graphs~\cite{MR4446253}, for sub-Riemannian Carnot groups \cite{MR4997323}, and for a number of other special cases in \cite{MR5039553}. As in the Riemannian case, a Euclidean (or more generally Hilbertian) factor is the typical source of non-rigidity. 

Rigidity for $L^p$-Wasserstein spaces of other exponents $p$ has also been studied: the real line is rigid for $p\neq2$~\cite{MR4127894}, while compact rank one symmetric spaces are rigid for all $p\in(1,\infty)$~\cite{MR4469075}. We refer to the introduction of~\cite{MR5039553} for an overview.

\vspace{0.5em}

\noindent \textbf{Structure of the paper.}
Section~2 recalls tangent cones in geodesic metric spaces, and introduces the
\emph{inner tangent cone} as the subcollection of directions represented by geodesics that extend
backward past their starting point; it also records some necessary definitions and
results from the theory of optimal transport and the geometry of Wasserstein spaces.

Building on Gigli~\cite{MR2847481}, Section~3 shows that the tangent spaces of
Otto's formal Riemannian calculus,
\[\mathrm{Tan}_\mu\mathscr P_2(M):=\overline{\{\nabla\varphi:\varphi\in C_c^\infty(M)\}}^{\,L^2_\mu},\]
are recognizable \emph{purely metrically} as the inner tangent cones of
$(\mathscr P_2(M),W_2)$. As a consequence, the dimension of $\mathrm{Tan}_\mu \mathscr P_2(M)$
is preserved under isometries, and since it encodes the size of the support of the
measure $\mu$, we infer that the set of empirical measures $\sum_{i=1}^n a_i\delta_{x_i}$
is also preserved. The proof of Theorem~\ref{B} (Section~5) is an immediate consequence.

In Section~4, we show that isometries also preserve the \emph{balanced} empirical
measures $\frac{1}{n}\sum_{i=1}^n \delta_{x_i}$, provided $\dim M \geq 2$. The remaining non-Euclidean one-dimensional case, the circle, is covered by the known
rigidity of $\mathscr P_2(S^1)$~\cite{MR4516798}.  Optimal transport between balanced
measures is attained at a permutation, so, up to rescaling, the set of $n$-point
balanced empirical measures is isometric to the \emph{Riemannian orbifold} $M^n/S_n$.

Section~5 builds on Lange~\cite{MR4163391} to lift isometries of $M^n/S_n$ to
isometries of $M^n$. By the uniqueness of the de~Rham
decomposition~\cite{MR1473665}, such a lift permutes the canonical indecomposable
factors of $M^n$; when $M$ has no Euclidean factor and $n \ge 3$, this pins it down to a diagonal
isometry of $M$ together with a block permutation invisible in $M^n/S_n$, so $\Phi$
is a push-forward and Theorem~\ref{main} follows. The presence of a Euclidean factor instead leaves
room for barycentric rotations, and tracking these gives the shape-preservation
Theorem~\ref{shape} (Section~6).

\section{Preliminaries}

Throughout, a distance-preserving map between metric spaces is called
\emph{isometric}, and a surjective isometric map is an \emph{isometry}.

\subsection{Tangent cones}\label{ss:cones}

For references for the following, see \cite[Section 5]{MR2117451} and
\cite{MR1835418}.
Let $(X,d)$ be a geodesic metric space and $p\in X$. Recall that a
\emph{geodesic} in a metric space is an isometric embedding of an interval, up
to a constant scaling factor; by contrast, \emph{geodesics} on a Riemannian
manifold keep their classical meaning of (merely locally minimizing) solutions
of the geodesic equation. For a geodesic $\gamma$ issuing from $p$ we
write $|\gamma'|$ for its speed.

For geodesics $\gamma_1:[0,a_1]\to X$ and $\gamma_2:[0,a_2]\to X$
with $\gamma_1(0)=\gamma_2(0)=p$ and $a_1,a_2>0$, we define
\begin{align}\label{geodesic_distance}
\delta(\gamma_1,\gamma_2):=\limsup_{t\to0}\,
\frac{d(\gamma_1(t),\gamma_2(t))}{t}.
\end{align}
Since the limsup is subadditive along triples of geodesics, $\delta$ is a
pseudo metric. The \emph{space of directions} $\mathrm{D}_pX$ at $p$ is the
metric space induced by this pseudo metric on the set of constant speed
geodesics issuing from $p$. The space $(\mathrm{D}_pX,\delta)$ may be
incomplete; the \emph{tangent cone} $\mathrm{C}_pX$ at $p$ arises as its
metric completion. We define the space of \emph{inner directions}
$\mathrm{InnD}_pX\subset\mathrm{D}_pX$ as the subspace consisting of directions represented
by geodesics $\gamma:[0,a]\to X$ which admit a \emph{local} extension to a geodesic
$\tilde\gamma:(-\epsilon,\epsilon)\to X$, $\epsilon>0$, i.e.\ such that
$\tilde\gamma|_{[0,\epsilon)}=\gamma|_{[0,\epsilon)}$. The \emph{inner tangent cone}
$\mathrm{InnC}_pX$ is the closure of $\mathrm{InnD}_pX$ in $\mathrm{C}_pX$;
equivalently, it is the metric completion of $\mathrm{InnD}_pX$,
isometrically embedded in $\mathrm{C}_pX$.

For non-constant $\gamma_1,\gamma_2$ as above and $t\in(0,a_1]$,
$s\in(0,a_2]$, the \emph{comparison angle}
$\tilde\theta(t,s)\in[0,\pi]$ is defined by
\[
\cos\big(\tilde\theta(t,s)\big)
=\frac{|\gamma_1'|^2t^2+|\gamma_2'|^2s^2-d^2(\gamma_1(t),\gamma_2(s))}
{2\,|\gamma_1'||\gamma_2'|\,ts},
\]
i.e.\ as the angle at the vertex corresponding to $p$ of a Euclidean triangle
with side lengths $|\gamma_1'|t$, $|\gamma_2'|s$, $d(\gamma_1(t),\gamma_2(s))$.
The \emph{upper} or \emph{Alexandrov angle} is
\[
\angle(\gamma_1,\gamma_2):=\limsup_{t,s\to0}\tilde\theta(t,s)\in[0,\pi].
\]
If the joint limit exists, the angle is said to exist \emph{in the strict
sense}. In that case the limit in \eqref{geodesic_distance} also exists and
\[
\delta(\gamma_1,\gamma_2)
=\Big(|\gamma_1'|^2+|\gamma_2'|^2
-2|\gamma_1'||\gamma_2'|\cos(\angle(\gamma_1,\gamma_2))\Big)^{\frac12}.
\]
A geodesic metric space $X$ is said to \emph{have angles} if the Alexandrov
angle between any pair of geodesic segments issuing from the same point
exists in the strict sense.

\subsection{Optimal transport and Wasserstein geometry}\label{ss:ot}

For a Borel map $f:X\to Y$ and $\mu\in\Prob(X)$, the push-forward
$f_\#\mu\in\Prob(Y)$ is defined by $f_\#\mu(E):=\mu(f^{-1}(E))$. Recall that the set of transport plans or \emph{couplings} between $\mu,\nu\in\Prob_2(M)$ is
\[
\Pi(\mu,\nu):=\big\{\pi\in\Prob(M\times M):\ \pi^1_\#\pi=\mu,\
\pi^2_\#\pi=\nu\big\},
\]
with $\pi^1,\pi^2$ the coordinate projections, so that
\begin{equation}\label{W}
	W_2^2(\mu,\nu)=\inf_{\pi\in\Pi(\mu,\nu)}\int_{M\times M}d^2(x,y)\,d\pi(x,y).
\end{equation}
The integral $\int_{M\times M}d^2(x,y)\,d\pi(x,y)$ is the \emph{cost} of the plan $\pi$. The plans for which the infimum in \eqref{W} is achieved are \emph{optimal}, and
$\Opt(\mu,\nu)\subset\Pi(\mu,\nu)$ denotes the set of these optimal plans; in this setting $\Opt(\mu,\nu)$ is always non-empty (cf. \cite[Theorem 4.1]{MR2459454}). A plan of the form $(\Id,T)_\#\mu$ is said
to be \emph{induced by the map} $T$.

Following convention (cf.\ \cite{MR2459454, MR2847481}) we set
$c(x,y):=\tfrac12\,d^2(x,y),$ and observe that a plan is optimal for the cost $c$ if and only if it is
optimal for $d^2$; all identities below refer to $c$.

We recall the following definitions from \cite{MR2847481}, see also \cite{MR1440931}:
\begin{itemize}
	\item A set $\Gamma\subset M\times M$ is \emph{$c$-cyclically monotone} if for every
$n\in\N$, every $(x_i,y_i)\in\Gamma$, $i=1,\dots,n$, and every permutation
$\sigma\in S_n$ it holds
\[
\sum_{i=1}^{n}c(x_i,y_i)\ \le\ \sum_{i=1}^{n}c(x_i,y_{\sigma(i)}).
\]
\item For $\psi:M\to\R\cup\{-\infty\}$, the \emph{$c^+$-transform}
$\psi^{c+}:M\to\R\cup\{-\infty\}$ is
\[
\psi^{c+}(x):=\inf_{y\in M}\big(c(x,y)-\psi(y)\big),
\]
and clearly
\begin{equation}\label{eq:weakdual}
\psi(x)+\psi^{c+}(y)\le c(x,y),\qquad\forall x,y\in M.
\end{equation}
A function $\varphi:M\to\R\cup\{-\infty\}$, not identically $-\infty$, is
\emph{$c$-concave} if $\varphi=\psi^{c+}$ for some $\psi$.
\item For $\varphi$ $c$-concave, the \emph{$c$-superdifferential} is defined as
\[
\partial^{c+}\varphi:=\big\{(x,y)\in M\times M:\
\varphi(x)+\varphi^{c+}(y)=c(x,y)\big\},
\]
and $\partial^{c+}\varphi(x):=\{y:(x,y)\in\partial^{c+}\varphi\}$.

\end{itemize}

\begin{theorem}[Characterizations of optimality,
{\cite[Theorems 1.2 and 1.6]{MR2847481}}]\label{thm:opt}
Let $\mu,\nu\in\Prob_2(M)$ and $\pi\in\Pi(\mu,\nu)$. The following are
equivalent:
\begin{enumerate}[(i)]
\item $\pi$ is optimal;
\item $\supp(\pi)$ is $c$-cyclically monotone;
\item $\supp(\pi)\subset\partial^{c+}\varphi$ for some $c$-concave function
$\varphi$.
\end{enumerate}
\end{theorem}

In particular, by the implication (iii)$\Rightarrow$(i) of
Theorem~\ref{thm:opt}, \emph{any} plan concentrated on the
$c$-superdifferential of a $c$-concave function is optimal between its own
marginals; this will be used repeatedly. A $c$-concave
function $\varphi$ such that $\partial^{c+}\varphi$ contains the support of
every optimal plan from $\mu$ to $\nu$ is called a \emph{Kantorovich
potential} for $(\mu,\nu)$; Kantorovich potentials always exist
\cite[Theorem 5.10]{MR2459454}.

As a special case of \cite[Lemma 2.9]{MR2847481}, for $\varphi\in C_c^\infty(M)$
and $\epsilon>0$ sufficiently small the function $\epsilon\varphi$ is
$c$-concave and
\[
\partial^{c+}(\epsilon\varphi)
=\big\{(x,\exp_x(-\epsilon\nabla\varphi(x))):x\in M\big\}.
\]
By the implication
(iii)$\Rightarrow$(i) of Theorem~\ref{thm:opt}, the plan
$(\Id,\exp(-\epsilon\nabla\varphi))_\#\sigma$ is then optimal for every
$\sigma\in\Prob_2(M)$. An inspection of the proof of
\cite[Lemma 2.9]{MR2847481} shows that the threshold for $\epsilon$ depends
only on the geometry of $M$, a compact set containing $\operatorname{supp}\varphi$, and upper
bounds for $\|\varphi\|_\infty$, $\|\nabla\varphi\|_\infty$ and
$\|\nabla^2\varphi\|_\infty$ (the supremum of the operator norm of the
Hessian), so it may be chosen uniformly over such families. We record this
uniform statement:

\begin{lemma}[Uniform smooth version of {\cite[Lemma 2.9]{MR2847481}}]\label{lem:unif}
Let $K\subset M$ be compact and $C>0$. There exists
$\epsilon_0=\epsilon_0(M,K,C)>0$ with the following property: for every
$\psi\in C_c^\infty(M)$ with
\[
\supp\psi\subset K,\qquad
\max\big\{\|\psi\|_\infty,\|\nabla\psi\|_\infty,\|\nabla^2\psi\|_\infty\big\}
\le C,
\]
and every $\epsilon\in(0,\epsilon_0]$, the function $\epsilon\psi$ is
$c$-concave and
\[
\partial^{c+}(\epsilon\psi)
=\big\{(x,\exp_x(-\epsilon\nabla\psi(x))):x\in M\big\}.
\]
In particular, $(\Id,\exp(-\epsilon\nabla\psi))_\#\sigma$ is an optimal plan
for every $\sigma\in\Prob_2(M)$, where $\exp(-\epsilon\nabla\psi)$ denotes the
map $x\mapsto\exp_x(-\epsilon\nabla\psi(x))$.
\end{lemma}

We close this subsection by recording the geodesic structure of the Wasserstein
space. The space $(\Prob_2(M),W_2)$ is geodesic, and its geodesics are the
\emph{displacement interpolations} (cf.\ \cite[Chapter 7]{MR2459454},
\cite{MR3050280}): writing $\mathrm{Geo}(M)$ for the set of geodesics
$\gamma:[0,1]\to M$ and $e_t:\mathrm{Geo}(M)\to M$, $e_t(\gamma):=\gamma(t)$, for
the evaluation maps, a curve $(\mu_t)_{t\in[0,1]}$ in $\Prob_2(M)$ is a geodesic
from $\mu_0$ to $\mu_1$ if and only if there exists
$\Lambda\in\Prob(\mathrm{Geo}(M))$ with
\[
   \mu_t=(e_t)_\#\Lambda\quad(t\in[0,1]),\qquad
   (e_0,e_1)_\#\Lambda\in\Opt(\mu_0,\mu_1);
\]
we call such a $\Lambda$ an \emph{optimal geodesic plan}: it records the geodesic
travelled by each portion of mass, while its endpoint marginal
$(e_0,e_1)_\#\Lambda$ is the induced optimal coupling of $\mu_0$ and $\mu_1$.

\begin{lemma}[Interior regularity, {\cite[Proposition~2.16]{MR3050280}}]\label{lem:georeg}
Let $(\mu_t)_{t\in[0,1]}$ be a geodesic in $(\Prob_2(M),W_2)$. For every
$t\in(0,1)$ there is a unique optimal plan in $\Opt(\mu_0,\mu_t)$, and it is
induced by a map from $\mu_t$; the optimal geodesic plan $\Lambda$ is likewise
unique. In particular $(\Prob_2(M),W_2)$ is non-branching.
\end{lemma}

In particular, if a geodesic $(\mu_t)_{t\in[0,1]}$ in $(\Prob_2(M),W_2)$ has a
finitely supported interior point $\mu_s=\sum_{i=1}^{n}a_i\delta_{x_i}$ at some
$s\in(0,1)$, with the $x_i$ distinct, then there are geodesics
$\gamma_1,\dots,\gamma_n\in\mathrm{Geo}(M)$ with $\gamma_i(s)=x_i$ and
$\mu_t=\sum_{i=1}^{n}a_i\delta_{\gamma_i(t)}$ for all $t\in[0,1]$; these
trajectories are pairwise distinct at every interior time, coinciding at most at
$t\in\{0,1\}$. In other words, such a geodesic cannot be continued beyond the
first collision of two trajectories.

\section{Infinitesimal structure of the Wasserstein space}

To motivate what follows, consider the discrete measure
$\mu=\sum_{i=1}^{n}a_i\delta_{x_i}$ with the $x_i$ distinct. Then for $\epsilon>0$ sufficiently small, \[B_\epsilon(x_1) \times \dots \times B_\epsilon(x_n) \to \mathscr P_2(M), \, (y_1,\dots, y_n) \mapsto \sum_{i=1}^n a_i \delta_{y_i},\] is an isometric embedding, after an appropriate rescaling of the metric on the balls. Therefore, since any geodesic in $B_\epsilon(x_1) \times \dots \times B_\epsilon(x_n)$ is two-sided around $(x_1,\dots, x_n) \in B_\epsilon(x_1) \times \dots \times B_\epsilon(x_n)$, we see that $T_{x_1}M \times \dots T_{x_n}M$ isometrically embeds into $\text{InnC}_\mu\mathscr P_2(M)$. 

Conversely by the above, any geodesic having $\mu$ as an interior point is of the form \[
   \mu_t:=\sum_{i=1}^{n}a_i\,\delta_{\gamma_i(t)},\qquad t\in(-\eta,\eta),
\] where $\eta>0$ and the $\gamma_i$ are geodesics in $M$. This shows that the embedding is surjective and thus we have shown that 
\[
   \mathrm{InnC}_\mu\Prob_2(M)\;\cong\;T_{x_1}M\times\cdots\times T_{x_n}M,
\]
a Euclidean space of dimension $\dim (M) \cdot n=\dim(M)\cdot|\supp\mu|$. 

In what follows we seek to determine the inner tangent cones for arbitrary measures, in possibly non-compact Riemannian manifolds $M$.

The infinitesimal geometry and tangent cones of $(\Prob_2(M),W_2)$ were studied by Gigli in
\cite{MR2847481} (see also \cite{MR3637959} and the recent works \cite{MR4918571, aussedat2026localitycentredtangentcones}); and the present section principally builds on this work. We begin by recalling notation and results from there.

Let $\mu\in\Prob_2(M)$,
by $\Prob_2(TM)_\mu$ we denote the set of Borel
probability measures $\gamma$ on the tangent bundle such that
$(\pi^M)_\#\gamma=\mu$ and
\[
\int|v|_x^2\,d\gamma(x,v)<\infty,
\]
where $\pi^M:TM\to M$ is the canonical projection. Elements of
$\Prob_2(TM)_\mu$ are called \emph{plans}, and the \emph{exponential} of a
plan $\gamma\in\Prob_2(TM)_\mu$ is $\exp_\mu(\gamma):=(\exp)_\#\gamma$. For
$\lambda\in\R$, the \emph{rescaling} of $\gamma$ by $\lambda$ is
$\lambda\cdot\gamma:=(\pi^M,\lambda\pi^1)_\#\gamma$, where $\pi^1(x,v):=v$.
With this notation, a curve $(\mu_t)_{t\in[0,1]}$ is a constant speed
geodesic from $\mu$ to $\nu$ if and only if there exists a plan
$\gamma\in\Prob_2(TM)_\mu$ with $\exp_\mu(\gamma)=\nu$ and
$\int|v|_x^2\,d\gamma=W_2^2(\mu,\nu)$ such that
\[
\mu_t=\exp_\mu(t\cdot\gamma),\qquad t\in[0,1];
\]
under these constraints the plan $\gamma$ is uniquely determined by the
geodesic \cite[Theorem 1.11]{MR2847481}. More generally, by rescaling the parameter to $[0,1]$ one associates to every
constant speed geodesic $(\mu_t)_{t\in[0,a]}$ issuing from $\mu$ a unique plan
$\gamma\in\Prob_2(TM)_\mu$, characterized by $\mu_t=\exp_\mu(t\cdot\gamma)$ for
$t\in[0,a]$; and two such geodesics $(\mu_t)$ and $(\tilde\mu_t)$, defined on
$[0,a]$ and $[0,\tilde a]$ and issuing from the same measure $\mu$, coincide on
$[0,\min\{a,\tilde a\}]$ if and only if their associated plans coincide
\cite[Proposition 1.12]{MR2847481}. We shall also use that for $t\in(0,1]$
the rescaled plan $t\cdot\gamma$ belongs to $\exp_\mu^{-1}(\mu_t)$, i.e.\
$(\pi^M,\exp)_\#(t\cdot\gamma)$ is optimal and
$\int|v|^2\,d(t\cdot\gamma)=W_2^2(\mu,\mu_t)$.

For $\gamma,\eta\in\Prob_2(TM)_\mu$ with disintegrations
$(\gamma_x)_{x\in M}$, $(\eta_x)_{x\in M}$ with respect to $\pi^M$, set
\[
W_\mu(\gamma,\eta):=\Big(\int_M W_2^2(\gamma_x,\eta_x)\,d\mu(x)\Big)^{1/2}.
\]
By \cite[Proposition 5.2]{MR2847481}, $W_\mu$ is a distance,
$(\Prob_2(TM)_\mu,W_\mu)$ is a Polish space, and with $T^2M:=\{(x,v_1,v_2):v_1,v_2\in T_xM\}$,
\begin{equation}\label{eq:Wmucoupling}
W_\mu^2(\gamma,\eta)=\inf\int|v_1-v_2|^2\,d\alpha(x,v_1,v_2),
\end{equation}
where the infimum (which is always attained) is taken over all \emph{admissible couplings} $\alpha\in\Prob(T^2M)$, characterized by
$(\pi^M,\pi^1)_\#\alpha=\gamma$ and $(\pi^M,\pi^2)_\#\alpha=\eta$.

The set $\DirB_\mu\subset\Prob_2(TM)_\mu$ is defined as
\[
\begin{aligned}
\DirB_\mu:=\big\{\gamma\in\Prob_2(TM)_\mu:\ &t\mapsto\exp_\mu(t\cdot\gamma)\\
&\text{is a geodesic on a right neighbourhood of }0\big\},
\end{aligned}
\]
and the \emph{geometric tangent space} $\TanB_\mu\Prob_2(M)$ is the closure
of $\DirB_\mu$ with respect to $W_\mu$ \cite[Definition 5.4]{MR2847481}. For
$\gamma\in\DirB_\mu$ we denote by
$e(\gamma)\in\mathrm{D}_\mu\Prob_2(M)$ the direction represented by the
geodesic $t\mapsto\exp_\mu(t\cdot\gamma)$.

By $L^2_\mu:=L^2_\mu(TM)$ we denote the Hilbert space of $L^2$ vector fields
on $M$ with respect to $\mu$, with inner product
\[
\langle v,w\rangle_\mu:=\int_M g_x\big(v(x),w(x)\big)\,d\mu(x).
\]
For $v\in L^2_\mu$ and $t\in\R$ we write $\exp(tv)$ for the map
$x\mapsto\exp_x(tv(x))$. The map $\iota_\mu:L^2_\mu\to\Prob_2(TM)_\mu$,
$v\mapsto(\Id,v)_\#\mu$ is isometric, since
$W_\mu^2(\iota_\mu v,\iota_\mu w)=\int|v-w|^2\,d\mu$; note
$\exp_\mu(t\cdot\iota_\mu(v))=\exp(tv)_\#\mu$. The \emph{formal tangent
space} is the `space of gradients'
\[
\mathrm{Tan}_\mu\Prob_2(M)
:=\overline{\{\nabla\varphi:\varphi\in C_c^\infty(M)\}}^{\,L^2_\mu}.
\]

In \cite{MR2847481} Gigli proved:
\begin{itemize}
\item The following three subsets of $L^2_\mu$ coincide
\cite[Corollary 6.4]{MR2847481} (no compactness of $M$ is required here,
cf.\ \cite[Section 6]{MR2847481}):
\begin{enumerate}[(i)]
\item $\mathrm{Tan}_\mu\Prob_2(M)$,
\item $\overline V^{\,L^2_\mu}$, where
\[
\begin{aligned}
V:=\big\{v\in L^2_\mu:\ &\exists\,\epsilon>0\text{ such that}\\
&(\Id,\exp(tv))_\#\mu\text{ is optimal for all }t\in[0,\epsilon]\big\},
\end{aligned}
\]
\item $\{v\in L^2_\mu:\ \iota_\mu(v)\in\TanB_\mu\Prob_2(M)\}$.
\end{enumerate}
\item If $M$ is compact: all upper angles between geodesics issuing from a
common measure exist in the strict sense \cite[Theorem 3.4]{MR2847481}, and
the natural map $e:(\DirB_\mu,W_\mu)\to(\mathrm{D}_\mu\Prob_2(M),\delta)$ is
an isometry, which therefore extends to an isometry $\overline e$
between $\TanB_\mu\Prob_2(M)$ and the tangent cone
$\mathrm{C}_\mu\Prob_2(M)$ \cite[Theorem 5.5]{MR2847481}.
\end{itemize}
For Wasserstein spaces over compact Alexandrov spaces
see Ohta \cite{MR2503990}.

In the sequel $M$ is only assumed to be complete and connected and no compactness is necessary. 

\begin{lemma}[Lipschitz estimate]\label{lem:lip}
Let $\mu\in\Prob_2(M)$ and $\gamma,\eta\in\Prob_2(TM)_\mu$. Then
\[
\limsup_{t\to 0}\,
\frac{W_2\big(\exp_\mu(t\cdot\gamma),\exp_\mu(t\cdot\eta)\big)}{t}
\ \le\ W_\mu(\gamma,\eta).
\]
In particular, if $\gamma,\eta\in\DirB_\mu$, then
$\delta(e(\gamma),e(\eta))\le W_\mu(\gamma,\eta)$.
\end{lemma}

\begin{proof}
Let $\alpha$ be an admissible coupling attaining the infimum in
\eqref{eq:Wmucoupling}. The map
$T^2M\ni(x,v_1,v_2)\mapsto(\exp_x(tv_1),\exp_x(tv_2))$ pushes $\alpha$ to a
transport plan between $\exp_\mu(t\cdot\gamma)$ and
$\exp_\mu(t\cdot\eta)$, so that for $t\neq0$
\[
\frac{W_2^2\big(\exp_\mu(t\cdot\gamma),\exp_\mu(t\cdot\eta)\big)}{t^2}
\le\int\Big(\frac{d\big(\exp_x(tv_1),\exp_x(tv_2)\big)}{t}\Big)^2
d\alpha(x,v_1,v_2).
\]
For a fixed $(x,v_1,v_2)$ the integrand converges to $|v_1-v_2|^2$ as
$t\to 0$, and is dominated by
$(|v_1|+|v_2|)^2\in L^1(\alpha)$. The claim follows by dominated convergence and
\eqref{eq:Wmucoupling}.
\end{proof}

\begin{lemma}[Isometric along gradient directions]\label{lem:grad}
Let $\mu\in\Prob_2(M)$ and $\varphi,\psi\in C_c^\infty(M)$. Then
\[
\lim_{t\to 0^+}\frac{W_2\big(\exp(t\nabla\varphi)_\#\mu,
\exp(t\nabla\psi)_\#\mu\big)}{t}
=\|\nabla\varphi-\nabla\psi\|_{L^2_\mu}.
\]
\end{lemma}

For compact $M$ this is a consequence of \cite[Theorem 5.5]{MR2847481}; the
point here is to record that it holds even without assuming compactness.

\begin{proof}
Write $\mu_t:=\exp(t\nabla\varphi)_\#\mu$, $\nu_t:=\exp(t\nabla\psi)_\#\mu$
and $\xi:=\varphi-\psi\in C_c^\infty(M)$. The inequality $\limsup\le$ follows
from Lemma \ref{lem:lip}. In particular we
may assume $\|\nabla\xi\|_{L^2_\mu}>0$, for otherwise the claim is already
shown.

It remains to show that $\liminf \geq$: set $H_t:=\exp(t\nabla\xi)$ and let $K\subset M$
be a compact set containing $\supp\varphi\cup\supp\psi$. By Lemma
\ref{lem:unif}, applied to $-\xi$, there is $t_0>0$ such that
$(\Id,H_t)_\#\sigma$ is an optimal plan for every $t\in(0,t_0]$ and every
$\sigma\in\Prob_2(M)$. Choosing $t_0$ smaller if necessary, the geodesic
$r\mapsto\exp_y(r\nabla\xi(y))$ is minimizing on $[0,t_0]$ for every
$y\in M$ (the injectivity radius is bounded below on a neighbourhood of $K$
and $\nabla\xi$ vanishes elsewhere). Thus the above plan has cost
$t^2\int|\nabla\xi|^2\,d\sigma=t^2\|\nabla\xi\|_{L^2_\sigma}^2$, and so
\[
W_2\big(\sigma,(H_t)_\#\sigma\big)=t\,\|\nabla\xi\|_{L^2_\sigma},
\qquad t\in(0,t_0],\ \sigma\in\Prob_2(M).
\]
Take $\sigma:=\nu_t$. The function $|\nabla\xi|^2$ is Lipschitz and vanishes
off $K$, and the displacement under $\exp(t\nabla\psi)$ is at most
$t\|\nabla\psi\|_\infty$; hence $\nu_t=\exp(t\nabla\psi)_\#\mu$ satisfies
\[
\|\nabla\xi\|_{L^2_{\nu_t}}^2=\int|\nabla\xi|^2\,d\nu_t
=\|\nabla\xi\|_{L^2_\mu}^2+O(t).
\]
Since $\|\nabla\xi\|_{L^2_\mu}>0$, and the fact that
$\sqrt{a^2+O(t)}=a+O(t)$, we see that\[
W_2\big(\nu_t,(H_t)_\#\nu_t\big)=t\,\|\nabla\xi\|_{L^2_{\nu_t}}
\ \ge\ t\,\|\nabla\xi\|_{L^2_\mu}-C_1t^2.
\]
Moreover, for each $x$ consider the curves
$t\mapsto H_t\big(\exp_x(t\nabla\psi(x))\big)$ and
$t\mapsto\exp_x(t\nabla\varphi(x))$: they coincide at $t=0$ and have there
the same derivative $\nabla\xi(x)+\nabla\psi(x)=\nabla\varphi(x)$; both are
constant for $x\notin K$, so by smoothness and compactness their distance is
bounded by $C_2t^2$, uniformly in $x$. Coupling $(H_t)_\#\nu_t$ with $\mu_t$
through $\mu$ thus gives $W_2\big((H_t)_\#\nu_t,\mu_t\big)\le C_2t^2$, and by
the triangle inequality
\[
W_2(\mu_t,\nu_t)\ \ge\ W_2\big(\nu_t,(H_t)_\#\nu_t\big)
-W_2\big((H_t)_\#\nu_t,\mu_t\big)\ \ge\ t\,\|\nabla\xi\|_{L^2_\mu}-(C_1+C_2)t^2,
\]
whence $\liminf_{t\to 0^+}W_2(\mu_t,\nu_t)/t\ge\|\nabla\xi\|_{L^2_\mu}$.
\end{proof}

\begin{proposition}[Two-sided geodesics]\label{geodesics}
Let $M$ be a connected, complete Riemannian manifold, let
$\mu\in\Prob_2(M)$ and $\varphi\in C_c^\infty(M)$. Set
\[
\mu_t:=(\exp(t\nabla\varphi))_\#\mu,
\]
then there exists $\eta>0$ so that $t\mapsto\mu_t$ is a constant speed
geodesic on $(-\eta,\eta)$.
\end{proposition}

\begin{proof}
Write $F_t:=\exp(t\nabla\varphi)$, i.e.\ $F_t(x)=\exp_x(t\nabla\varphi(x))$, so
that $\mu_t=(F_t)_\#\mu$. Let $K\subset M$ be a compact set containing
$\supp\varphi$. Since $M$ is
complete, $K$ is compact and $\nabla\varphi$ is bounded, we can choose
$\eta_0>0$ such that for every $x\in M$ the Riemannian geodesic
$\gamma_x:\R\to M$, $r\mapsto\exp_x(r\nabla\varphi(x))$, is minimizing on
$[-\eta_0,\eta_0]$. By a standard argument we can choose
$\eta_1\in(0,\eta_0)$ such that $F_r$ is a smooth diffeomorphism of $M$ for
every $|r|<2\eta_1$; clearly $F_r=\Id$ outside $K$, and so $F_r(K)=K$.
For $|r|\le\eta_1$ define $\varphi_r\in C^\infty_c(M)$ by
\[
\varphi_r(F_r(x)):=\varphi(x)+\frac r2|\nabla\varphi(x)|^2,
\]
such that $\supp\varphi_r\subset F_r(K)=K$. Using the first variation formula, we show:
\begin{equation}\label{eq:firstvar}
\nabla\varphi_r(F_r(x))=\gamma_x'(r)
=\frac{d}{dt}\Big|_{t=r}\exp_x(t\nabla\varphi(x)).
\end{equation}
Since $\gamma_x$
is a geodesic,
$\gamma_x(t)=\exp_{\gamma_x(s)}\big((t-s)\gamma_x'(s)\big)$, and
\eqref{eq:firstvar} gives, for all $y\in M$ and $|s|,|t|\le\eta_1$,
\begin{equation}\label{eq:flow}
F_t\circ F_s^{-1}(y)=\exp_y\big((t-s)\nabla\varphi_s(y)\big).
\end{equation}

The map $(r,x)\mapsto(r,F_r(x))$ is a smooth bijection of
$(-2\eta_1,2\eta_1)\times M$ with everywhere invertible differential, so its
inverse is smooth; consequently $(r,y)\mapsto\varphi_r(y)$ is smooth on
$[-\eta_1,\eta_1]\times M$ and
\[
C:=\sup_{|r|\le\eta_1}\Big(\|\varphi_r\|_\infty+\|\nabla\varphi_r\|_\infty
+\|\nabla^2\varphi_r\|_\infty\Big)<\infty,
\]
the suprema being finite because all functions are supported in the fixed
compact set $K$. Let $\epsilon_0=\epsilon_0(M,K,C)$ be given by Lemma
\ref{lem:unif} and set $\eta_2:=\min\{\epsilon_0,\eta_1\}$,
$\eta:=\eta_2/2$.

Set $S:=\|\nabla\varphi\|_{L^2_\mu}$. Since each $\gamma_x$ is minimizing on
$[-\eta_0,\eta_0]$ with speed $|\nabla\varphi(x)|$ and
$(F_s,F_t)_\#\mu$ is a transport plan, we know that for $-\eta<s\le t<\eta$,
\begin{equation}\label{eq:upper}
W_2(\mu_s,\mu_t)\le\left(\int d^2(F_s,F_t)\, d\mu\right)^{\frac{1}{2}}=(t-s)S.
\end{equation}
Moreover, for $0<\eta'\le\eta$, applying Lemma \ref{lem:unif} to
$\psi:=-\varphi_{-\eta'}$ with
$\epsilon:=2\eta'\le\eta_2\le\epsilon_0$ and
$\sigma:=\mu_{-\eta'}\in\Prob_2(M)$, we obtain that
\[
\big(\Id,\exp(2\eta'\nabla\varphi_{-\eta'})\big)_\#\mu_{-\eta'}
\overset{\eqref{eq:flow}}{=}
\big(\Id,F_{\eta'}\circ F_{-\eta'}^{-1}\big)_\#\big((F_{-\eta'})_\#\mu\big)
=(F_{-\eta'},F_{\eta'})_\#\mu
\]
is optimal. Thus equality holds in \eqref{eq:upper} for
$(s,t)=(-\eta',\eta')$. For arbitrary $-\eta<s\le t<\eta$, pick
$\eta'\in(\max\{|s|,|t|\},\eta)$; then by \eqref{eq:upper} and the triangle
inequality
\[
2\eta'S=W_2(\mu_{-\eta'},\mu_{\eta'})
\le W_2(\mu_{-\eta'},\mu_s)+W_2(\mu_s,\mu_t)+W_2(\mu_t,\mu_{\eta'})
\le 2\eta'S,
\]
forcing equality in each term, in particular
$W_2(\mu_s,\mu_t)=(t-s)S$. Hence $t\mapsto\mu_t$ is a constant speed geodesic
on $(-\eta,\eta)$.
\end{proof}

\begin{remark}
A direct computation from \eqref{eq:firstvar} shows that $(\varphi_r)$ solves
the Hamilton--Jacobi equation
$\partial_r\varphi_r+\tfrac12|\nabla\varphi_r|^2=0$, i.e.\ it is the
evolution of $\varphi$ under the Hopf--Lax semigroup; cf.\
\cite[Theorem 2.18]{MR3050280} and \cite{MR2920736}, and \cite{MR2358290}
for related computations.
\end{remark}

Thus we are finally able to prove the following:

\begin{proposition}\label{InnT}
Let $M$ be a complete, connected Riemannian manifold and
$\mu\in\Prob_2(M)$. Then the assignment
$\nabla\varphi\mapsto e\big(\iota_\mu(\nabla\varphi)\big)$,
$\varphi\in C_c^\infty(M)$, extends to a natural isometry
$\mathrm{Tan}_\mu\Prob_2(M)\ \cong\ \mathrm{InnC}_\mu\Prob_2(M).$

\end{proposition}

\begin{proof}
By Proposition \ref{geodesics} and Lemma \ref{lem:grad}, the composition
\begin{equation}\label{eq:isom}
\{\nabla\varphi:\varphi\in C_c^\infty(M)\}
\ \xrightarrow{\ \iota_\mu\ }\ \DirB_\mu
\ \xrightarrow{\ e\ }\ \mathrm{D}_\mu\Prob_2(M)
\end{equation}
is well defined and isometric, and its image is contained in
$\mathrm{InnD}_\mu\Prob_2(M)$. Since $\mathrm{Tan}_\mu\Prob_2(M)$ is the
$L^2_\mu$-closure of $\{\nabla\varphi\}$ and $\mathrm{C}_\mu\Prob_2(M)$ is
complete, \eqref{eq:isom} extends to an isometric embedding
$\overline e:\mathrm{Tan}_\mu\Prob_2(M)\to\mathrm{C}_\mu\Prob_2(M)$ with
closed image contained in the closure of
$\mathrm{InnD}_\mu\Prob_2(M)$, i.e.\ in $\mathrm{InnC}_\mu\Prob_2(M)$.

Conversely, since the image of $\overline e$ is closed and
$\mathrm{InnC}_\mu\Prob_2(M)$ is the closure of
$\mathrm{InnD}_\mu\Prob_2(M)$, it suffices to show that the image of
$\overline e$ contains the direction represented by any constant speed
geodesic $(\mu_t)_{t\in(-\epsilon,\epsilon)}$ with $\mu_0=\mu$.

Let $\gamma\in\Prob_2(TM)_\mu$ be the plan associated to
$(\mu_t)_{t\in[0,\epsilon)}$ by \cite[Proposition 1.12]{MR2847481}, so that
$\gamma\in\DirB_\mu$, $t\cdot\gamma\in\exp_\mu^{-1}(\mu_t)$ for
$t\in(0,\epsilon)$, and the direction in question is $e(\gamma)$. We claim
that $\gamma=\iota_\mu(v)$ for some $v\in L^2_\mu$ (cf.\ the interior
regularity results in \cite[Section 2.3]{MR3050280}). The case of a constant geodesic being
trivial, we assume that the speed $S$ of $(\mu_t)$ is positive. Set
$\epsilon':=\epsilon/2$, write $\gamma^+:=\gamma$ and let
$\gamma^-\in\Prob_2(TM)_\mu$ be the plan associated to the reversed
geodesic $[0,\epsilon)\ni t\mapsto\mu_{-t}$, so that
$\epsilon'\cdot\gamma^\pm\in\exp_\mu^{-1}(\mu_{\pm\epsilon'})$ and, $\int|w|^2\,d\gamma^\pm=S^2$.

Let $\alpha:=\int_M\gamma^-_x\otimes\gamma^+_x\,d\mu(x)\in\Prob(T^2M)$ be the
admissible coupling of $\gamma^-$ and $\gamma^+$ given by the product of the
disintegrations. As in the proof of Lemma \ref{lem:lip}, the map
$(x,u,w)\mapsto(\exp_x(\epsilon'u),\exp_x(\epsilon'w))$ pushes $\alpha$ to
a transport plan between $\mu_{-\epsilon'}$ and $\mu_{\epsilon'}$, and so,
by the triangle and the Minkowski inequalities,
\[
\begin{aligned}
2\epsilon'S=W_2(\mu_{-\epsilon'},\mu_{\epsilon'})
&\le\Big(\int d^2\big(\exp_x(\epsilon'u),\exp_x(\epsilon'w)\big)\,
d\alpha\Big)^{\frac12}\\
&\le\epsilon'\big\||u|+|w|\big\|_{L^2(\alpha)}
\le2\epsilon'S,
\end{aligned}
\]
where the last step used that $\||u|\|_{L^2(\alpha)}=\||w|\|_{L^2(\alpha)}=S$.
Equality therefore holds throughout. The Minkowski inequality step gives $|u|=|w|$
$\alpha$-a.e.; the triangle inequality step gives
$d(\exp_x(\epsilon'u),\exp_x(\epsilon'w))=\epsilon'|u|+\epsilon'|w|$, and $x$
lies in the interior of a minimizing geodesic joining $\exp_x(\epsilon'u)$
and $\exp_x(\epsilon'w)$. Its two halves are $r\mapsto\exp_x(ru)$ and $r\mapsto\exp_x(rw)$; being
minimizing, hence smooth, this geodesic has no corner at the interior point
$x$, which implies $w=-u$ (trivially if $u=w=0$). Thus $\gamma^-_x\otimes\gamma^+_x$ is concentrated on $\{w=-u\}$
for $\mu$-a.e.\ $x$, forcing both $\gamma^\pm_x$ to be Dirac. Hence
$\gamma=\iota_\mu(v)$ for some $v\in L^2_\mu$.

In particular $\iota_\mu(v)=\gamma\in\DirB_\mu\subset\TanB_\mu\Prob_2(M)$,
so $v\in\mathrm{Tan}_\mu\Prob_2(M)$ by \cite[Corollary 6.4]{MR2847481}, and there are
$\varphi_n\in C^\infty_c(M)$ with $\nabla\varphi_n\to v$ in $L^2_\mu$. By
Lemma \ref{lem:lip} and since $\iota_\mu$ is isometric,
\[
\delta\Big(e(\gamma),\,e\big(\iota_\mu(\nabla\varphi_n)\big)\Big)
\le W_\mu\big(\iota_\mu(v),\iota_\mu(\nabla\varphi_n)\big)
=\|v-\nabla\varphi_n\|_{L^2_\mu}\to 0,
\]
hence $e(\gamma)$ belongs to the closed image of $\overline e$. As $\gamma$
was chosen arbitrarily,
\[
\mathrm{InnD}_\mu\Prob_2(M)\subset\overline e(\mathrm{Tan}_\mu\Prob_2(M)),
\]
and so taking closures we obtain
$\mathrm{InnC}_\mu\Prob_2(M)=\overline e(\mathrm{Tan}_\mu\Prob_2(M))$.
\end{proof}

\begin{remark}
The space $\mathrm{Tan}_\mu\Prob_2(M)=\overline{\{\nabla\varphi:\varphi\in
C_c^\infty(M)\}}^{\,L^2_\mu}$ is the tangent space underlying Otto's formal
Riemannian calculus on $\Prob_2(M)$ \cite{Otto}, where a tangent vector at
$\mu$ is a velocity potential $\varphi$ acting through the continuity equation
and the metric tensor is $\int_M\langle\nabla\varphi_1,\nabla\varphi_2\rangle\,
d\mu$. Proposition~\ref{InnT} thus describes this tangent space \emph{purely
metrically}, as the inner tangent cone $\mathrm{InnC}_\mu\Prob_2(M)$,
reconstructed from the infinite dimensional metric space $(\Prob_2(M),W_2)$ with no reference to gradients or the continuity equation.
\end{remark}

For a geodesic metric space $X$, the \emph{Euclidean rank}
$\operatorname{rank}(X)$ is the supremum of all $n\in\N$ for which there is an
isometric embedding $\mathbb E^n\hookrightarrow X$.

\begin{corollary}\label{dim}
Let $\mu\in\Prob_2(M)$. Then
\[
\operatorname{rank}\big(\mathrm{InnC}_\mu\Prob_2(M)\big)=\dim(M)\cdot|\supp\mu|.
\]
\end{corollary}
\begin{proof}
By Proposition~\ref{InnT}, $\mathrm{InnC}_\mu\Prob_2(M)$ is isometric to
$\mathrm{Tan}_\mu\Prob_2(M)$, a closed subspace of $L^2_\mu$ and hence a Hilbert
space, whose Euclidean rank equals its dimension. Gradients of bump functions
realize arbitrary tangent vectors at the distinct points of $\supp\mu$, so the
dimension is $\dim(M)\cdot|\supp\mu|$. Hence
$\operatorname{rank}\big(\mathrm{InnC}_\mu\Prob_2(M)\big)=\dim(M)\cdot|\supp\mu|.$
\end{proof}


\section{Preservation of discrete measures}

For coefficients $a_1,\dots,a_n\in[0,1]$ with $\sum_{i=1}^n a_i=1$, let
$\Delta(a_1,\dots,a_n)\subset\Prob_2(M)$ denote the set of measures
$\sum_{i=1}^n a_i\delta_{x_i}$, $x_i\in M$. The set of \emph{balanced
combinations} is $\Delta_n:=\Delta(\tfrac1n,\dots,\tfrac1n)$, and the set of
general $n$-point convex combinations is
\[
\tilde\Delta_n:=\bigcup_{\substack{a_1,\dots,a_n\in[0,1]\\ \sum_i a_i=1}}
\Delta(a_1,\dots,a_n)=\{\mu\in\Prob_2(M):|\supp\mu|\le n\}.
\]
We show that if $M$ is not isometric to the real line, every isometry of $\Prob_2(M)$ preserves $\Delta_n$. We first treat the general
convex combinations, where no assumption on $M$ is required.

\begin{theorem}\label{wp}
Let $\Phi:\Prob_2(M)\to\Prob_2(N)$ be an isometry. Then
$\dim(M)=\dim(N)$ and $\Phi(\tilde\Delta_n(M))=\tilde\Delta_n(N)$ for all
$n\in\N$.
\end{theorem}
\begin{proof}
An isometry maps geodesics issuing from $\mu$ to geodesics issuing
from $\Phi(\mu)$, preserving $\delta$ and two-sided extendability; it
therefore induces an isometry
$\mathrm{InnC}_\mu\Prob_2(M)\cong\mathrm{InnC}_{\Phi(\mu)}\Prob_2(N)$, and
so in particular the Euclidean rank of the inner tangent cones is preserved. Thus by Corollary~\ref{dim},
\[
\dim(M)\cdot|\supp\mu|=\dim(N)\cdot|\supp\Phi(\mu)|\qquad
\text{for all }\mu\in\Prob_2(M).
\]
Since $\inf_\mu|\supp\mu|=\inf_\mu|\supp\Phi(\mu)|=1$, taking infima of both sides gives $\dim(M)=\dim(N)$; cancelling this
common factor then yields $|\supp\mu|=|\supp\Phi(\mu)|$ for all $\mu$, i.e.\
$\Phi(\tilde\Delta_n(M))=\tilde\Delta_n(N)$ for all $n$.
\end{proof}

For $\mu\in\Prob_2(M)$ and $n\in\N$ set
\[
\mathrm{C}_n(\mu):=\big\{\nu\in\Prob_2(M):\ \text{some geodesic in }
\tilde\Delta_n\text{ joins }\mu\text{ to }\nu\big\}.
\]
By the interior regularity of geodesics
(cf. \cite[Proposition 2.16]{MR3050280}), $\mathrm{C}_n(\mu)=\emptyset$ when
$|\supp\mu|>n$, and if \(\mu=\sum_{i=1}^n a_i\delta_{x_i}\) has exactly \(n\) distinct atoms, then
\(\mathrm C_n(\mu)\subset \Delta(a_1,\dots,a_n)\).

A subset $A\subset B$ of a metric space is \emph{weakly convex relative to
$B$} if any two points of $A$ are joined by a geodesic contained in $B$.

\begin{lemma}\label{convex}
Let $M$ be a complete connected Riemannian manifold with $\dim M\ge2$ and
$\mu=\sum_{i=1}^n a_i\delta_{x_i}\in\tilde\Delta_n(M)\setminus
\tilde\Delta_{n-1}(M)$. Then $\mathrm{C}_n(\mu)$ is weakly convex
relative to $\tilde\Delta_n(M)$ if and only if
$a_1=\dots=a_n=\tfrac1n$.
\end{lemma}
\begin{proof}
If $a_1=\dots=a_n=\tfrac1n$ then $\mathrm{C}_n(\mu)=\Delta_n$, which is weakly
convex: two balanced measures are joined by an optimal-matching
interpolation, which stays in $\Delta_n\subset\tilde\Delta_n$.

Conversely, assume the $a_i$ are not all equal. Choose a pair $(i,j)$
attaining $\min\{d(x_k,x_l):a_k\neq a_l\}$ (which exists, as not all
$a_i$ are equal), set $D:=d(x_i,x_j)>0$, fix a minimizing geodesic from
$x_i$ to $x_j$ and let $m$ be its midpoint. Let $u\in T_mM$ be the unit vector
with $x_i=\exp_m(\tfrac D2 u)$ and $x_j=\exp_m(-\tfrac D2 u)$. Being an
interior point of a minimizing geodesic, $m$ lies in the cut locus of neither
$x_i$ nor $x_j$; hence $f_i:=d(\cdot,x_i)$ and $f_j:=d(\cdot,x_j)$ are smooth
near $m$, with $\nabla f_i(m)=-u$ and $\nabla f_j(m)=u$.

We first claim that $d(m,x_k)>D/2$ for every $k\neq i,j$. Suppose to the
contrary that $d(m,x_k)\leq D/2$. As $a_i\neq a_j$, either $a_k\neq a_i$ or
$a_k\neq a_j$; say $a_k\neq a_i$. Then minimality of $D$ gives
\[
D\le d(x_i,x_k)\le d(x_i,m)+d(m,x_k)\le\tfrac D2+\tfrac D2=D,
\]
so all inequalities are equalities: $d(m,x_k)=D/2$ and $m$ lies on a
minimizing geodesic from $x_i$ to $x_k$. Its first half is a minimizing
geodesic from $x_i$ to $m$, hence equals
the chosen geodesic as $m\notin\operatorname{Cut}(x_i)$. The geodesic from $x_i$ to $x_k$ thus prolongs the one
from $x_i$ to $x_j$. Since geodesics in Riemannian manifolds do not branch, this forces $x_k=x_j$, a contradiction. By finiteness there
is therefore $\delta>0$ with $d(m,x_k)\ge D/2+2\delta$ for all $k\neq i,j$.

Since $\dim M\ge2$, we can choose a unit vector $b\in T_mM$ with $b\perp u$. For clarity, set
$\alpha:=\tfrac12$, and for $\epsilon>0$ set $y_k=y_k'=x_k$ for $k\neq i,j$ and
\[
\begin{aligned}
y_i&=\exp_m\!\big(\epsilon(\alpha u+b)\big), &\qquad
y_i'&=\exp_m\!\big(\epsilon(\alpha u-b)\big),\\
y_j&=\exp_m\!\big(\epsilon(-\alpha u-b)\big), &\qquad
y_j'&=\exp_m\!\big(\epsilon(-\alpha u+b)\big),
\end{aligned}
\]
and $\nu:=\sum_k a_k\delta_{y_k}$, $\nu':=\sum_k a_k\delta_{y_k'}$.

We first claim that \emph{$\nu,\nu'\in\mathrm{C}_n(\mu)$.} Indeed by the first-order expansion of the
smooth functions $f_i,f_j$ at $m$,
\[
d(y_i,x_i)=\tfrac D2-\alpha\epsilon+O(\epsilon^2),\qquad
d(y_i,x_j)=\tfrac D2+\alpha\epsilon+O(\epsilon^2),
\]
so $d(y_i,x_i)<d(y_i,x_j)$ for small $\epsilon$; and since
$d(m,y_i)\le\epsilon\sqrt{\alpha^2+1}<\delta$ for small $\epsilon$, for every
$k\neq i,j$
\[
d(y_i,x_k)\ge d(m,x_k)-d(m,y_i)>\tfrac D2+\delta>d(y_i,x_i).
\]
Hence $x_i$ is the nearest atom to $y_i$; the same holds for $y_i'$, and
symmetrically $x_j$ is nearest to $y_j,y_j'$, while $y_k=x_k$ for $k\neq i,j$.
In other words, $d(y_l,x_l)\le d(y_l,x_k)$ and $d(y_l',x_l)\le d(y_l',x_k)$ for all
$k,l$, and for any transport plan $\pi=(\pi_{kl})$ from $\mu$ to $\nu$ (i.e.\
$\sum_l\pi_{kl}=a_k$, $\sum_k\pi_{kl}=a_l$),
\[
\sum_{k,l}\pi_{kl}\,d(x_k,y_l)^2\ \ge\ \sum_{k,l}\pi_{kl}\,d(x_l,y_l)^2
=\sum_l a_l\,d(x_l,y_l)^2 .
\]
Thus the diagonal plan $\tau=\sum_k a_k\delta_{(x_k,y_k)}$ is optimal and its
displacement interpolation is a geodesic in $\tilde\Delta_n$ from $\mu$ to
$\nu$; hence $\nu\in\mathrm{C}_n(\mu)$, and likewise $\nu'\in\mathrm{C}_n(\mu)$.

Next we show that \emph{$\nu,\nu'$ are not joined within $\tilde\Delta_n$}: in normal
coordinates at $m$ one has
$d(\exp_m(\epsilon p),\exp_m(\epsilon q))^2=\epsilon^2|p-q|^2+O(\epsilon^4)$,
uniformly for $p,q$ ranging over the four vectors above. Since
$|(\alpha u+b)-(-\alpha u+b)|^2=4\alpha^2$ and
$|(\alpha u+b)-(\alpha u-b)|^2=4$, we obtain for small $\epsilon$:
\begin{equation}\label{eq:swap}
d(y_i,y_j')^2+d(y_j,y_i')^2=8\alpha^2\epsilon^2+O(\epsilon^4)
\ <\ 8\epsilon^2+O(\epsilon^4)=d(y_i,y_i')^2+d(y_j,y_j')^2.
\end{equation}

Suppose a geodesic from $\nu$ to $\nu'$ stayed in $\tilde\Delta_n$. In that case no atom
can split so the optimal coupling is induced by a weight-preserving bijection of
the atoms of $\nu$ onto those of $\nu'$. Since $a_i\neq a_j$ forbids
$y_i\mapsto y_j'$ and $y_j\mapsto y_i'$, any bijection other than the identity ($y_i\mapsto y_i'$, $y_j\mapsto y_j'$, $x_k\mapsto x_k$) sends a near
atom to some $x_k$ or permutes the $x_k$ hence costs $\geq c>0$ for some constant $c$ independent of
$\epsilon$; and so only the identity could be optimal.
But, taking $a_i<a_j$ (the other case being symmetric), the splitting coupling
\[
a_i\delta_{(y_i,y_j')}+a_i\delta_{(y_j,y_i')}+(a_j-a_i)\delta_{(y_j,y_j')}
+\sum_{k\neq i,j}a_k\delta_{(x_k,x_k)},
\]
is cheaper than the identity since by \eqref{eq:swap},
\[
a_i\big(d(y_i,y_i')^2+d(y_j,y_j')^2-d(y_i,y_j')^2-d(y_j,y_i')^2\big)
=6a_i\epsilon^2+O(\epsilon^4)>0.
\]
 Hence no bijection is optimal, so $\nu,\nu'$ are not joined
within $\tilde\Delta_n$ and $\mathrm{C}_n(\mu)$ is not weakly convex.
\end{proof}

\begin{theorem}\label{balanced}
Let $\Phi:\Prob_2(M)\to\Prob_2(M)$ be an isometry, and suppose
$M$ is not isometric to the real line. Then $\Phi(\Delta_n(M))=\Delta_n(M)$
for all $n\in\N$.
\end{theorem}
\begin{proof}
\emph{Case $\dim M=1$:} A complete connected $1$-dimensional Riemannian
manifold not isometric to $\R$ is a rescaled circle $S^1$. By
\cite[Theorem 4.2]{MR4516798},
$(\Prob_2(S^1),W_2)$ is isometrically rigid, so $\Phi=\varphi_\#$ for
some $\varphi\in\operatorname{Isom}(M)$, and as a consequence $\Phi(\Delta_n(M))=\Delta_n(M)$.

\emph{Case $\dim M\ge2$:} By Theorem~\ref{wp}, $\Phi$ is an
isometry with $\Phi(\tilde\Delta_k(M))=\tilde\Delta_k(M)$ for every
$k\in\N$. Fix $n$ and a measure
$\mu\in\tilde\Delta_n(M)\setminus\tilde\Delta_{n-1}(M)$, i.e.\ with exactly
$n$ distinct atoms. As an isometry preserving $\tilde\Delta_n(M)$, $\Phi$
carries geodesics of $\tilde\Delta_n(M)$ issuing from $\mu$ to geodesics of
$\tilde\Delta_n(M)$ issuing from $\Phi(\mu)$; hence
$\Phi(\mathrm{C}_n(\mu))=\mathrm{C}_n(\Phi(\mu))$, and $\mathrm{C}_n(\mu)$
is weakly convex relative to $\tilde\Delta_n(M)$ if and only if
$\mathrm{C}_n(\Phi(\mu))$ is. Since $\Phi$ also preserves
$\tilde\Delta_{n-1}(M)$, the measure $\Phi(\mu)$ likewise has exactly $n$
distinct atoms, so Lemma~\ref{convex} applies to both $\mu$ and $\Phi(\mu)$
and yields
\begin{align*}
	\mu\in\Delta_n(M)\ &\Longleftrightarrow\ \mathrm{C}_n(\mu)\text{ weakly convex}
\\ \ &\Longleftrightarrow\ \mathrm{C}_n(\Phi(\mu))\text{ weakly convex}
\ \Longleftrightarrow\ \Phi(\mu)\in\Delta_n(M).
\end{align*}

Thus $\Phi$ restricts to a bijection of
$\Delta_n(M)\setminus\tilde\Delta_{n-1}(M)$ onto itself. This set is dense in
$\Delta_n(M)$, and therefore the claim follows by density.\end{proof}

The \emph{exotic isometries} of $\Prob_2(\mathbb E)$ constructed in
\cite{MR2731158} break the balanced
combinations $\Delta_n$ while preserving the general $n$-point combinations
$\tilde\Delta_n$; Theorem~\ref{balanced} shows that this mechanism is
unavailable as soon as $M\not\cong\mathbb E$.

\section{Orbifolds and the proofs of Theorems A and C}
Theorem~\ref{B} is a direct corollary of Theorem~\ref{wp}:

\begin{proof}[Proof of Theorem~\ref{B}]
Since any isometry $\varphi:M\to N$ induces an isometry $\varphi_\#$ between the Wasserstein spaces, one direction is clear.

For the other direction, let $\Phi:(\mathscr P_2(M),W_2)\to (\mathscr P_2(N),W_2)$ be an isometry between the Wasserstein spaces over Riemannian manifolds $M$ and $N$ respectively. By Theorem~\ref{wp}, the dimensions of these manifolds are equal and the isometry restricts to an isometry between $\Delta_1(M)\cong M$ and $\Delta_1(N)\cong N$, proving the other direction.
\end{proof}

The proof of the main theorem relies on Thurston'ï¿½ï¿½s notion of Riemannian orbifolds and Lytchak's metric characterization of these, as developed by Lange in \cite{MR4163391}, see also \cite{MR2719410}. For further background on orbifolds in general we refer the reader to the original  \cite{thurston1979geometry}.

For a Riemannian manifold $M$, the symmetric group $S_n$ acts isometrically on
the Riemannian direct product $M^n$ by permutation of the entries. The metric
quotient $M^n/S_n$ is a Riemannian orbifold, and we denote the natural projection
map by $\pi_n:M^n \to M^n/S_n$. Within this context, we first show the following
auxiliary lemma:

\begin{lemma}\label{orbilift}
Let $M$ be a connected complete Riemannian manifold and $n\in\mathbb N$.
Every metric isometry $\gamma: M^n/S_n \to M^n/S_n$ lifts to an
isometry $\overline{\gamma}:M^n\to M^n$ such that
$\pi_n\circ \overline{\gamma}=\gamma\circ \pi_n.$
\end{lemma}

\begin{proof}
If $\dim M=0$, then $M$ is a point and the claim is immediate. Hence we may assume
that $m:=\dim M>0$. Let
$p:\widetilde M\longrightarrow M$
be the universal Riemannian covering and let $\Gamma=\operatorname{Deck}(p)$.

As a composition of coverings, the map
$P:=\pi_n\circ p^n:\widetilde M^n\longrightarrow M^n/S_n$ is an orbifold covering. Since
$\widetilde M^n$ is simply connected, it is the universal orbifold covering (cf. Thurston's \cite[Theorem 2.9]{MR4163391}); and
$
        \operatorname{Deck}(P)=\Gamma^n\rtimes S_n .
$

Thus we have the general picture
\[
        \widetilde M^n/\Gamma^n\cong M^n,\qquad
        \widetilde M^n/(\Gamma^n\rtimes S_n)\cong M^n/S_n .
\]
In what follows, for the action of $\Gamma^n\rtimes S_n$ on $\widetilde M^n$, we use the convention
\[
(\alpha_1,\ldots,\alpha_n;\sigma)\cdot(x_1,\ldots,x_n)
=
\bigl(\alpha_1x_{\sigma^{-1}(1)},\ldots,
      \alpha_nx_{\sigma^{-1}(n)}\bigr).
\]

Since $M$ is complete, the orbifold $M^n/S_n$ is complete. A metric isometry of a
complete Riemannian orbifold is a submetry with discrete fibres, hence it is an orbifold
covering in the sense of Thurston by \cite[Theorem~1.2]{MR4163391}. Therefore the map
$\gamma\circ P:\widetilde M^n\to M^n/S_n$ is again a universal orbifold covering.

By
the universal property of universal orbifold coverings there exists an isometry
$
        \widetilde\gamma:\widetilde M^n\to \widetilde M^n
$
such that
$
        P\circ\widetilde\gamma=\gamma\circ P.
$
As a consequence of this, $\widetilde\gamma$ normalizes $\Gamma^n\rtimes S_n$. Indeed, for
$d\in \Gamma^n\rtimes S_n$, we have
$
P\circ \widetilde\gamma d\widetilde\gamma^{-1}
=
\gamma\circ P\circ d\circ\widetilde\gamma^{-1}
=
\gamma\circ P\circ\widetilde\gamma^{-1}
=
P,
$
and hence
$
        \widetilde\gamma d\widetilde\gamma^{-1}
        \in \operatorname{Deck}(P)=\Gamma^n\rtimes S_n .
$
Applying the same argument to $\widetilde\gamma^{-1}$ gives
$
        \widetilde\gamma(\Gamma^n\rtimes S_n)\widetilde\gamma^{-1}
        =
        \Gamma^n\rtimes S_n.
$

To complete the proof it remains to show that $\widetilde\gamma$ also normalizes
$\Gamma^n$. We may assume $n\geq2$, the case $n=1$ being trivial.

The group $\Gamma^n$ is generated by the \emph{single-block} elements
$\delta_a(\beta):=(e,\ldots,e,\beta,e,\ldots,e;\,e)$, with $\beta\in\Gamma$ in
the $a$-th component ($1\leq a\leq n$); each acts as the identity on every block
$b\neq a$. Call an element of $\Gamma^n\rtimes S_n$ \emph{block-permuting} if its
$S_n$-component is non-trivial. The idea is to consider a quantity $\nu$ on
$\operatorname{Isom}(\widetilde M^n)$, invariant under conjugation by isometries,
with
\begin{equation}\label{eq:nu-strict}
        \nu\bigl(\delta_a(\beta)\bigr)>\nu(d),\,
        \text{for every block-permuting }d\in\Gamma^n\rtimes S_n \text{ and } \beta\in \Gamma.
\end{equation}
Granting \eqref{eq:nu-strict}, the conjugate
$\widetilde\gamma\,\delta_a(\beta)\,\widetilde\gamma^{-1}\in\Gamma^n\rtimes S_n$
has the same $\nu$-value as $\delta_a(\beta)$, so it is not block-permuting and
hence lies in $\Gamma^n$. As the $\delta_a(\beta)$ generate $\Gamma^n$, this
gives $\widetilde\gamma\Gamma^n\widetilde\gamma^{-1}\subseteq\Gamma^n$, and
applying the same argument to $\widetilde\gamma^{-1}$ yields equality.

By the de~Rham decomposition theorem \cite{MR1473665}, write
$\widetilde M=\mathbb R^{\ell}\times N_1\times\cdots\times N_q$, with
$\mathbb R^{\ell}$ the Euclidean factor and each $N_i$ irreducible and non-flat;
correspondingly
\[
        \widetilde M^n=\mathbb R^{n\ell}\times\prod_{a=1}^n\prod_{i=1}^q N_i^{(a)},
\]
where $N_i^{(a)}$ denotes the copy of $N_i$ in the $a$-th block. Every isometry
of the simply connected $\widetilde M^n$ preserves the Euclidean factor and
permutes the irreducible factors $\{N_i^{(a)}\}$.

\emph{Case 1: $\widetilde M$ is not flat ($q\geq1$).} For
$g\in\operatorname{Isom}(\widetilde M^n)$ let $\nu(g)$ be the number of factors
$N_i^{(a)}$ on which $g$ acts as the identity. This is conjugation-invariant: any
isometry $h$ merely permutes the factors $\{N_i^{(a)}\}$, and $hgh^{-1}$ acts as
the identity on $h(N_i^{(a)})$ if and only if $g$ acts as the identity on
$N_i^{(a)}$; hence the two counts coincide. A single-block element
$\delta_a(\beta)$ is the identity on every block $b\neq a$, hence on the
$(n-1)q$ factors which lie in those blocks, so $\nu(\delta_a(\beta))\geq(n-1)q$. A
block-permuting $d=(\eta_1,\ldots,\eta_n;\sigma)$ carries the factors of block
$b$ into block $\sigma(b)$, so it fixes a factor only inside a block with
$\sigma(b)=b$; as $\sigma\neq e$ fixes at most $n-2$ indices,
$\nu(d)\leq(n-2)q<(n-1)q$, which is \eqref{eq:nu-strict}.

\emph{Case 2: $\widetilde M$ is flat ($\widetilde M=\mathbb R^m$).} Now every
isometry $g$ of $\widetilde M^n=\mathbb R^{nm}$ is affine; let
$\nu(g):=\dim\ker(L_g-\operatorname{Id})$ be the multiplicity of the eigenvalue
$1$ of its linear part $L_g\in O(nm)$. This is conjugation-invariant because
$L_{hgh^{-1}}=L_h\,L_g\,L_h^{-1}$ is conjugate to $L_g$ and hence has the same
eigenvalue-$1$ multiplicity. For $\delta_a(\beta)$ with $\beta=(A,t)$, the linear
part is $A$ on block $a$ and the identity on the other $n-1$ blocks, so
\[
        \nu(\delta_a(\beta))=\dim\ker(A-\operatorname{Id})+(n-1)m\geq(n-1)m+1;
\]
here $\dim\ker(A-\operatorname{Id})\geq1$, since otherwise $\operatorname{Id}-A$
would be invertible and $\beta$ would fix the point $(\operatorname{Id}-A)^{-1}t$,
contradicting the freeness of $\Gamma$.

For the upper bound, let $d=(\alpha_1,\ldots,\alpha_n;\sigma)$ be block-permuting
and write $\alpha_a=(D_a,t_a)$ with $D_a\in O(m)$. Its linear part acts on
$\mathbb R^{nm}=(\mathbb R^m)^n$ by
$L_d(x_1,\ldots,x_n)
        =\bigl(D_1x_{\sigma^{-1}(1)},\ldots,D_nx_{\sigma^{-1}(n)}\bigr)$. 
Thus $x=(x_1,\ldots,x_n)$ is fixed if and only if $x_a=D_ax_{\sigma^{-1}(a)}$ for
every $a$. These equations couple the coordinates within each cycle of $\sigma$:
along a cycle $(a_1\,a_2\,\cdots\,a_\ell)$ (i.e.\ with $\sigma(a_t)=a_{t+1}$), they give
$x_{a_1}=D_{a_1}x_{a_{\ell}}= \dots = D_{a_1} D_{a_\ell}\cdots D_{a_{t+1}}\,x_{a_t}$ and, on closing the loop,
$
        x_{a_1}=\bigl(D_{a_1}D_{a_\ell}\cdots D_{a_2}\bigr)\,x_{a_1}.
$
Thus $x_{a_1}$ lies in the eigenvalue-$1$ subspace of the cycle product
$D_{a_1}D_{a_\ell}\cdots D_{a_2}\in O(m)$. Thus has dimension at most $m$, and the choice of $x_{a_1}$ determines $x_{a_2},\ldots,x_{a_\ell}$. Each of the $c(\sigma)$ cycles of
$\sigma$ therefore contributes at most $m$ to the fixed subspace, so
\[
        \nu(d)=\dim\ker(L_d-\operatorname{Id})\ \leq\ m\,c(\sigma)\ \leq\ (n-1)m,
\]
the last inequality because a permutation $\sigma\neq e$ has at most $n-1$ cycles. Hence
$\nu(\delta_a(\beta))\geq(n-1)m+1>(n-1)m\geq\nu(d)$, which is
\eqref{eq:nu-strict}.

In either case \eqref{eq:nu-strict} is established, and therefore
$\widetilde\gamma\Gamma^n\widetilde\gamma^{-1}=\Gamma^n$.

It follows that $\widetilde\gamma$ descends to an isometry $\overline{\gamma}:\widetilde M^n/\Gamma^n\to \widetilde M^n/\Gamma^n$, defined by $\overline{\gamma}([x]_{\Gamma^n})=[\widetilde\gamma x]_{\Gamma^n}$. It is readily verified that $
        \pi_n\circ\overline{\gamma}=\gamma\circ\pi_n$
as claimed.
\end{proof}

\begin{proof}[Proof of Theorem~\ref{main}]
\emph{Non-rigidity.} If $M$ splits a non-trivial Euclidean factor, then
$(\mathscr P_2(M),W_2)$ admits isometries not induced by isometries of $M$
\cite[Theorem~A]{MR5039553}.

\emph{Rigidity.} Suppose conversely that $M$ has no non-trivial Euclidean factor. If $M$ is a
point there is nothing to prove, so assume $M$ is non-trivial. Then
$M\not\cong\mathbb E$, so by Theorem~\ref{balanced} every Wasserstein isometry
$\Phi:\mathscr P_2(M)\to\mathscr P_2(M)$ satisfies $\Phi(\Delta_n)=\Delta_n$ for
all $n\in\mathbb N$.

Equip $M^n$ with the rescaled product distance
$d_n^2((x_i),(y_i))=\frac1n\sum_{i=1}^n d^2(x_i,y_i)$, and $M^n/S_n$ with the
induced quotient distance. Let $\pi_n:M^n\to M^n/S_n$ be the projection. For
balanced empirical measures the optimal transport problem is a linear
programming problem over the Birkhoff polytope, so its optimum is attained at a
permutation matrix \cite[Chapter 2]{COTFNT}; hence
\[
W_2^2\Big(\tfrac1n\textstyle\sum_{i=1}^n\delta_{x_i},
\tfrac1n\sum_{i=1}^n\delta_{y_i}\Big)
=\min_{\sigma\in S_n}\tfrac1n\sum_{i=1}^n d^2(x_i,y_{\sigma(i)}).
\]
Thus $J_n:M^n/S_n\to\Delta_n$,
$J_n([(x_1,\ldots,x_n)])=\frac1n\sum_{i=1}^n\delta_{x_i}$, is an isometry, and
$\Phi$ induces the isometry $\phi_n:=J_n^{-1}\circ\Phi|_{\Delta_n}\circ J_n$ of
$M^n/S_n$. For $n=1$ this is an isometry $f:=\phi_1$ of $M$, characterized by
$\Phi(\delta_x)=\delta_{f(x)}$. We claim that
\[
\phi_n([(x_1,\ldots,x_n)])=[(f(x_1),\ldots,f(x_n))]\qquad\text{for every }n\geq3.
\]

Fix $n\geq3$. The constant rescaling does not change the isometry group, so
Lemma~\ref{orbilift} lifts $\phi_n$ to an isometry $G_n:M^n\to M^n$ with
$\pi_n\circ G_n=\phi_n\circ \pi_n$. Since $J_n(\pi_n(x,\ldots,x))=\delta_x$, we
have $\phi_n(\pi_n(x,\ldots,x))=\pi_n(f(x),\ldots,f(x))$. The diagonal point
$(f(x),\ldots,f(x))$ is the only element of its $\pi_n$-fibre, so
$G_n(x,\ldots,x)=(f(x),\ldots,f(x))$. Consequently
$H_n:=(f^{-1},\ldots,f^{-1})\circ G_n$
is an isometry of $M^n$ that descends to $M^n/S_n$ and fixes the total diagonal
pointwise: $H_n(x,\ldots,x)=(x,\ldots,x)$.

We now use uniqueness of the metric product decomposition. Since $M$ has no
Euclidean factor, de~Rham's theorem in the non-simply-connected form of
Eschenburg--Heintze \cite{MR1473665} gives
\[
M=M_1\times\cdots\times M_p,
\]
with each $M_r$ non-trivial, indecomposable, and not isometric to $\mathbb E$;
this decomposition is unique up to the order of the factors, the corresponding
factor foliations being canonically determined (see \cite{MR2399098} for more details and a
generalization to finite-dimensional geodesic metric spaces). Because these
foliations are canonical, any isometry of a finite product of such factors
permutes them, interchanging only mutually isometric factors and acting on each
by an isometry (cf. \cite{MR2399098}).

As $M^n=\prod_{a=1}^n\prod_{r=1}^p M_r$ again has no Euclidean factor, $H_n$
permutes its $np$ indecomposable factors in this way. Write $z=(z_1,\ldots,z_n)$
with $z_a=(z_{a,1},\ldots,z_{a,p})$. For each output coordinate $(a,r)$ there are
an input coordinate $(b,s)$ with $M_s\cong M_r$ and an isometry
$\eta_{a,r}:M_s\to M_r$ such that $(H_n(z))_{a,r}=\eta_{a,r}(z_{b,s})$. Evaluate on
the diagonal $z_1=\cdots=z_n=x$: since $H_n(x,\ldots,x)=(x,\ldots,x)$, we get
$x_r=\eta_{a,r}(x_s)$ for all $x\in M$. As distinct factors vary independently and
each $M_r$ is non-trivial, this forces $s=r$ and
$\eta_{a,r}=\operatorname{id}_{M_r}$. Thus there are permutations $\beta_r\in S_n$
with
\[
(H_n(z))_{a,r}=z_{\beta_r(a),r};
\]
a priori, $H_n$ may permute the $n$ copies of each factor $M_r$ independently.

It remains to show that the $\beta_r$ are all equal; we may assume $p\geq2$, so
$\dim M\geq2$ (for $p=1$ there is nothing to prove). Since $H_n$ descends to
$M^n/S_n$, each conjugate $H_n\pi H_n^{-1}$ ($\pi\in S_n$) descends to the
identity, hence fixes every $S_n$-orbit; and an orbit-fixing isometry lies in
$S_n$, because it restricts on the configuration space $U$ of $n$ distinct
points of $M$ to a deck transformation of the free covering $U\to U/S_n$, and
$U$ is connected (as $\dim M\geq2$), so this is a single $\sigma\in S_n$ and the
isometry equals $\sigma$ on $\overline U=M^n$. Thus $H_n\pi H_n^{-1}\in S_n$.

From $(H_n^{-1}w)_{a,r}=w_{\beta_r^{-1}(a),r}$ we get
$(H_n\pi H_n^{-1}w)_{a,r}=w_{(\beta_r^{-1}\pi^{-1}\beta_r)(a),\,r}$; as this is a
single block permutation, $\beta_r^{-1}\pi\beta_r$ is independent of $r$, so
$\beta_s\beta_r^{-1}\in Z(S_n)=\{e\}$ for $n\geq3$.

Therefore $\beta_1=\cdots=\beta_p=:\beta$, and
$H_n(z_1,\ldots,z_n)=(z_{\beta(1)},\ldots,z_{\beta(n)})$ merely permutes the $n$
blocks, which is invisible in $M^n/S_n$. This proves the claim:
\[
\Phi\Big(\tfrac1n\textstyle\sum_{i=1}^n\delta_{x_i}\Big)
=\tfrac1n\sum_{i=1}^n\delta_{f(x_i)}\qquad(n\geq3).
\]
Finally, let $\mu=\sum_{j=1}^k\frac{m_j}{N}\delta_{x_j}$ be a finitely supported
rational measure ($m_j\in\mathbb N$, $\sum_j m_j=N$). Choosing $N\geq3$, we may
view $\mu$ as an element of $\Delta_N$ by repeating each $x_j$ exactly $m_j$
times; hence $\Phi(\mu)=\sum_{j=1}^k\frac{m_j}{N}\delta_{f(x_j)}=f_\#\mu$. These
measures are dense in $(\mathscr P_2(M),W_2)$, and $\Phi$ and $f_\#$ are
isometries, so $\Phi=f_\#$. Thus every isometry of $\mathscr P_2(M)$ is induced
by an isometry of $M$. Together with the non-rigidity case, this proves the
theorem.
\end{proof}

\begin{remark}
The hypothesis $n\geq3$ enters the proof only through $Z(S_n)=\{e\}$, but it is
genuinely needed there: for $n=2$ one has $Z(S_2)=S_2$, and the step forcing
$H_n$ to be a block permutation fails. Indeed, taking $M=M_1\times M_2$, the
isometry
\[
   T\bigl((x_1,x_2),(y_1,y_2)\bigr)=\bigl((x_1,y_2),(y_1,x_2)\bigr)
\]
of $M^2$ descends to an orbifold isometry of $M^2/S_2$ that fixes the diagonal,
yet does not arise from a permutation of the two blocks.
\end{remark}

\section{Shape preservation over a split Euclidean factor}

In this section we prove Theorem~\ref{shape}. Recall that an isometry $\Phi$ of
$\Prob_2(M)$ is \emph{shape-preserving} if for every $\mu$ there is
$\gamma_\mu\in\operatorname{Isom}(M)$ with $\Phi(\mu)=(\gamma_\mu)_\#\mu$, and
\emph{exotic} otherwise.

Throughout, $M\not\cong\E$. By the de~Rham decomposition write
$M=N\times\E^k$, where $\E^k$ is the maximal Euclidean factor and $N$ is free of
Euclidean factors. If $k=0$ then $M$ has no Euclidean factor and
Theorem~\ref{shape} is exactly the rigid case of Theorem~\ref{main}. We thus assume $k\ge1$; since $M\not\cong\E$ this implies
$\dim M\ge2$.

For $\mu\in\Prob_2(M)$ let $\mu_{\E^k}:=(\pi_{\E^k})_\#\mu$ denote the marginal of $\mu$ on the
Euclidean factor and
\[
   \bary(\mu):=\int_{\E^k}x\,d\mu_{\E^k}(x)\in\E^k
\]
its \emph{barycentre}, which is well defined and unique because $\E^k$ is Hilbert and
$\mu$ has finite second moment. For $R\in O(k)$ let
$r_{R,b}:M\to M$, $r_{R,b}(x_N,x_E)=(x_N,\,b+R(x_E-b))$, be the rotation by $R$
about $b\in\E^k$ in the Euclidean factor, and define the \emph{barycentre
rotation}
\[
   \Psi_R(\mu):=\big(r_{R,\bary(\mu)}\big)_\#\mu .
\]
Each $\Psi_R$ is an isometry of $\Prob_2(M)$ that fixes every Dirac mass, and
$R\mapsto\Psi_R$ is an injective homomorphism $O(k)\to\operatorname{Isom}(\Prob_2(M))$;
this is the content of the non-rigidity construction of \cite[Theorem~A]{MR5039553}
(see also \cite{MR2731158} for the original case of $N$ trivial). The maps $\Psi_R$ with
$R\neq\mathrm{id}$ are shape-preserving but not push-forwards. Our task is to show
that these, together with $\operatorname{Isom}(M)$, exhaust
$\operatorname{Isom}(\Prob_2(M))$.

We record two facts about isometries of $M^n=N^n\times\E^{kn}$. The first is a
direct consequence of the de~Rham decomposition which we already used in proof of Lemma~\ref{orbilift}: since $N$, and hence $N^n$, has
no Euclidean de~Rham factor, $\E^{kn}$ is the maximal Euclidean de~Rham factor of
$M^n$; by its uniqueness this factor is canonical and preserved by every isometry. Consequently every isometry of $M^n$ splits
factorwise,
\[
   \operatorname{Isom}(M^n)=\operatorname{Isom}(N^n)\times\operatorname{Isom}(\E^{kn})
   =\operatorname{Isom}(N^n)\times\big(O(kn)\ltimes\R^{kn}\big).
\]

The second fact concerns the
Euclidean part $\E^{kn}$, to which we now turn.

Embed $S_n\hookrightarrow O(kn)$ as the group $P=\{P_\sigma:\sigma\in S_n\}$ of
\emph{block permutations}
$P_\sigma(v_1,\dots,v_n)=(v_{\sigma^{-1}(1)},\dots,v_{\sigma^{-1}(n)})$,
$v_i\in\E^k$, and split orthogonally: 

$$\E^{kn}=U\oplus V,$$
\begin{align*}
   U:=\{(v,\dots,v):v\in\E^k\}\cong\E^k, \quad
   V:=\Big\{(v_i):\textstyle\sum_i v_i=0\Big\}\cong\E^{k(n-1)}.
\end{align*}

For $B\in O(k)$ let $\Theta_B$ be the linear isometry of $\E^{kn}$ fixing $U$
pointwise and acting on $V$ as $B$ in each coordinate (under
$V\cong\E^k\otimes\rho_{\mathrm{std}}$, $\Theta_B=B\otimes 1$). Note
$\Theta_B$ is the linear part of the barycentre rotation restricted to balanced
configurations.

\begin{lemma}\label{lem:normalizer}
Let $g\in\operatorname{Isom}(\E^{kn})=O(kn)\ltimes\R^{kn}$
normalize the block-permutation subgroup $P$ and fix $U$ pointwise. Then there
exist $B\in O(k)$ and $\tau\in S_n$ with $g=P_\tau\circ\Theta_B$.
\end{lemma}

\begin{proof}
We may assume $n\ge2$. Write $g(x)=Lx+t$ with $L\in O(kn)$. Since
$0\in U$ and $g$ fixes $U$ pointwise, $t=g(0)=0$, so $g=L$ is a linear isometry
that fixes $U$ pointwise and normalizes $P$. In particular $L(U)=U$, and since
$V=U^{\perp}$ also $L(V)=V$.

Decompose $\E^{kn}=\E^k\otimes\E^n$ with $P$ acting as $1\otimes\rho_{\mathrm{perm}}$.
Over $\E$, $\rho_{\mathrm{perm}}=\rho_{\mathrm{triv}}\oplus\rho_{\mathrm{std}}$ with
$\rho_{\mathrm{std}}$ the $(n-1)$-dimensional standard representation, which is
absolutely irreducible; under this identification $U=\E^k\otimes\rho_{\mathrm{triv}}$
and $V=\E^k\otimes\rho_{\mathrm{std}}$ are the $P$-isotypic components. Since $L$
normalizes $P$, conjugation by $L$ induces
$\alpha\in\operatorname{Aut}(S_n)$ with $LP_\sigma L^{-1}=P_{\alpha(\sigma)}$;
restricted to $V$, where $P_\sigma$ acts as $1\otimes\rho_{\mathrm{std}}(\sigma)$,
this shows $L|_V$ conjugates $\rho_{\mathrm{std}}$ to $\rho_{\mathrm{std}}\circ\alpha$,
whence $\rho_{\mathrm{std}}\cong\rho_{\mathrm{std}}\circ\alpha$. Therefore $\alpha$ is
inner, which was automatic for $n\ne6$ (where $\operatorname{Out}(S_n)=1$), while for $n=6$
the outer automorphism swaps transpositions with triple transpositions
\cite{MR660917}, on which the character $\#\mathrm{Fix}(\cdot)-1$ of
$\rho_{\mathrm{std}}$ takes the values $3$ and $-1$. It therefore does not fix
$\rho_{\mathrm{std}}$ and is excluded.

Thus there exists $\tau \in S_n$ such that
$P_{\alpha(\sigma)}=P_{\tau}P_\sigma P_{\tau}^{-1}$
for all $\sigma$; thus $P_\tau^{-1}L$ centralizes $P$, and restricted to $V$,
$P|_V=1\otimes\rho_{\mathrm{std}}(S_n)$. As $\rho_{\mathrm{std}}$ is absolutely
irreducible, Schur's lemma gives $\operatorname{End}_{S_n}(V)=\operatorname{End}(\E^k)$,
so $(P_\tau^{-1}L)|_V=B\otimes1$ for some $B\in O(k)$. Since $L$ and $P_\tau$ both
fix $U$ pointwise, $(P_\tau^{-1}L)|_U=\mathrm{id}_U$. Hence
$L=P_\tau\circ\big(\mathrm{id}_U\oplus(B\otimes1)\big)=P_\tau\circ\Theta_B$, i.e.\
$g=P_\tau\circ\Theta_B$.
\end{proof}

\begin{proof}[Proof of Theorem~\ref{shape}]
As noted, we may assume $k\ge1$, hence $\dim M\ge2$. Let
$\Phi\in\operatorname{Isom}(\Prob_2(M))$. By Theorem~\ref{balanced}, $\Phi(\Delta_n)=\Delta_n$ for all $n$.
As in the proof of Theorem~\ref{main}, $J_n:M^n/S_n\to\Delta_n$ is an isometry,
$\phi_n:=J_n^{-1}\circ\Phi|_{\Delta_n}\circ J_n\in\operatorname{Isom}(M^n/S_n)$,
and $f:=\phi_1\in\operatorname{Isom}(M)$ satisfies $\Phi(\delta_x)=\delta_{f(x)}$.

Fix $n\ge\max\{3,k+1\}$. By Lemma~\ref{orbilift}, $\phi_n$ lifts to
$G_n\in\operatorname{Isom}(M^n)$, and
$H_n:=(f^{-1},\dots,f^{-1})\circ G_n$ is an isometry of $M^n$ that descends to
$M^n/S_n$ and fixes the total diagonal pointwise (as in the proof of
Theorem~\ref{main}). By the de~Rham splitting above, $H_n=(H_n^N,H_n^E)$ with
$H_n^N\in\operatorname{Isom}(N^n)$ and
$H_n^E\in O(kn)\ltimes\R^{kn}$.

Again exactly as in the proof of Theorem~\ref{main} (using that the
configuration space of $n$ distinct points of $M$ is connected, which holds as
$\dim M\ge2$), we see that $H_n$ normalizes $S_n\leq O(nk)$.
Since $S_n$ acts diagonally on the product $M^n=N^n\times\E^{kn}$, $H_n^E$ and $H_n^N$ also both normalize $S_n$ and fix the diagonal in $\E^{kn}$ and $N^n$ respectively. Therefore, we can apply Lemma~\ref{lem:normalizer} to get,
\[H_n^E=P_{\tau}\circ\Theta_{B_n}, \, \text{ for some }B_n\in O(k),\ \tau\in S_n.\]

Since $N$ has no Euclidean factor, the argument in the
proof of Theorem~\ref{main}, applied to $N^n=\prod_{a,r}N_r$, shows that
$H_n^N$ is a pure block permutation $P_\beta$ of the $n$ copies of $N$.

The barycentre rotation $\Theta_{B_n}$ commutes
with every $P_\sigma$. Hence, comparing the two factors of the descent
condition $H_n\pi H_n^{-1}\in S_n$ gives
$\beta\pi\beta^{-1}=\tau\pi\tau^{-1}$ for all $\pi$, so
$\tau^{-1}\beta\in Z(S_n)=\{e\}$; thus $\beta=\tau=:\sigma$ and
\[
   H_n=P_\sigma\circ\big(\mathrm{id}_{N^n}\times\Theta_{B_n}\big).
\]

A block permutation of the $n$ copies of $M$ is invisible in $M^n/S_n$, so for a
balanced measure $\mu=\tfrac1n\sum_i\delta_{x_i}\in\Delta_n$, unwinding
$G_n=(f,\dots,f)\circ H_n$ yields
\begin{equation}\label{eq:shapen}
   \Phi(\mu)=f_\#\,\Psi_{B_n}(\mu),\qquad \mu\in\Delta_n .
\end{equation}

Finally, we need to show independence of $n$: if $n\mid N$ then $\Delta_n\subset\Delta_N$ (repeat
each atom $N/n$ times), and \eqref{eq:shapen} for $n$ and for $N$ both compute
$\Phi$ on $\mu\in\Delta_n$, so $\Psi_{B_n}(\mu)=\Psi_{B_N}(\mu)$. Choosing $\mu$
whose fluctuations $x_i^E-\bary(\mu)$ span $\E^k$ and have pairwise distinct
norms (and therefore trivial $O(k)$-stabilizer; generic for $n\ge\max\{3,k+1\}$), this
equality places $B_N^{-1}B_n$ in that stabilizer, hence $B_n=B_N$. Thus for all
$m,n\ge\max\{3,k+1\}$, $B_m=B_{mn}=B_n\equiv :B$ is independent of $n$.

\emph{Conclusion:} By \eqref{eq:shapen}, $\Phi=f_\#\circ\Psi_B$ on
$\bigcup_{N\ge\max\{3,k+1\}}\Delta_N$, which contains the finitely supported
rational-weight measures and is therefore dense in $\Prob_2(M)$; and so by density, $\Phi=f_\#\circ\Psi_B$. In particular $\Phi$ is
shape-preserving.

Finally, $R\mapsto\Psi_R$ embeds $O(k)$ into
$\operatorname{Isom}(\Prob_2(M))$ with $\Psi_R$ a push-forward only for
$R=\mathrm{id}$, and for $\gamma\in\operatorname{Isom}(M)$ one has
$\gamma_\#\Psi_R\gamma_\#^{-1}=\Psi_{\rho(\gamma)R\rho(\gamma)^{-1}}$, where
$\rho(\gamma)\in O(k)$ is the rotational part of $\gamma$ on the Euclidean
factor. Hence $(\gamma,R)\mapsto\gamma_\#\circ\Psi_R$ is an isomorphism
$\operatorname{Isom}(M)\ltimes O(k)\xrightarrow{\ \cong\ }
\operatorname{Isom}(\Prob_2(M))$. This proves Theorem~\ref{shape}.
\end{proof}

\begin{remark}
The argument degenerates precisely at $M\cong\E$, and for two independent
reasons: Theorem~\ref{balanced} requires $M\not\cong\E$ (on the line the balanced
combinations $\Delta_n$ are not preserved), and the descent step needs the
configuration space of $n$ distinct points to be connected, which fails in
dimension one. This is consistent with Kloeckner's exotic flow on
$\Prob_2(\E)$, which both unbalances $\Delta_n$ and exploits the disconnectedness
of ordered configurations of points on a line.
\end{remark}
\newpage
\section*{Acknowledgements}
I am grateful to Alexander Lytchak and Nicola Gigli for their comments and suggestions; I am particularly indebted to the former for his encouragement and reading a late version of this article, and to the latter for answering questions on  his work~\cite{MR2847481}. Thanks are also due to Stefan K\"uhnlein for instructive  comments, and Maximilian Wackenhuth for a fruitful discussion.

\bibliographystyle{alpha}
\bibliography{paper}

\end{document}